\documentclass{amsart}
\usepackage[utf8]{inputenc}
\usepackage[maxnames=4,maxalphanames=4,style=alphabetic-verb]{biblatex}
\bibliography{main.bib}
\usepackage[margin=1in]{geometry}
\usepackage{graphicx}
\usepackage{amsmath,amsfonts,amssymb,amsthm,amsaddr,etoolbox}
\usepackage{latexsym,hyperref}
\usepackage{xcolor}
\usepackage{scrextend}
\usepackage{enumitem}
\usepackage[normalem]{ulem}
\usepackage{mathtools}
\usepackage[linewidth=1pt]{mdframed}

\author[Mittal and Kuber]{Mihir Mittal and Amit Kuber}
\address{Department of Mathematics and Statistics\\Indian Institute of Technology, Kanpur\\ Uttar Pradesh, India}
\email{mihirmittal24@gmail.com, askuber@iitk.ac.in}
\title{Exponentiable linear orders need not be transitive}
\keywords{linear orders, transitive, exponentiable, cyclically transitive}
\subjclass[2020]{06A05}

\newcommand\restr[2]{{
  \left.\kern-\nulldelimiterspace 
  #1 
  \littletaller 
  \right|_{#2} 
  }}

\theoremstyle{plain}
\newtheorem{defn}{Definition}[section]

\newtheorem{prop}[defn]{Proposition}
\newtheorem*{prop*}{Proposition}
\newtheorem*{thm*}{Theorem}
\newtheorem{thm}[defn]{Theorem}
\newtheorem{cor}[defn]{Corollary}

\newtheorem*{claim*}{Claim}

\theoremstyle{remark}
\newtheorem{rem}[defn]{Remark}
\theoremstyle{remark}

\theoremstyle{remark}

\theoremstyle{remark}
\newtheorem{exmp}[defn]{Example}
\theoremstyle{remark}

\theoremstyle{remark}

\theoremstyle{remark}

\theoremstyle{remark}

\theoremstyle{remark}
\newtheorem{ques}[defn]{Question}
\numberwithin{equation}{section}

\makeatletter
\let\save@mathaccent\mathaccent
\newcommand*\if@single[3]{%
  \setbox0\hbox{${\mathaccent"0362{#1}}^H$}%
  \setbox2\hbox{${\mathaccent"0362{\kern0pt#1}}^H$}%
  \ifdim\ht0=\ht2 #3\else #2\fi
  }
\newcommand*\rel@kern[1]{\kern#1\dimexpr\macc@kerna}
\newcommand*\widebar[1]{\@ifnextchar^{{\wide@bar{#1}{0}}}{\wide@bar{#1}{1}}}
\newcommand*\wide@bar[2]{\if@single{#1}{\wide@bar@{#1}{#2}{1}}{\wide@bar@{#1}{#2}{2}}}
\newcommand*\wide@bar@[3]{%
  \begingroup
  \def\mathaccent##1##2{%
    \let\mathaccent\save@mathaccent
    \if#32 \let\macc@nucleus\first@char \fi
    \setbox\z@\hbox{$\macc@style{\macc@nucleus}_{}$}%
    \setbox\tw@\hbox{$\macc@style{\macc@nucleus}{}_{}$}%
    \dimen@\wd\tw@
    \advance\dimen@-\wd\z@
    \divide\dimen@ 3
    \@tempdima\wd\tw@
    \advance\@tempdima-\scriptspace
    \divide\@tempdima 10
    \advance\dimen@-\@tempdima
    \ifdim\dimen@>\z@ \dimen@0pt\fi
    \rel@kern{0.6}\kern-\dimen@
    \if#31
      \overline{\rel@kern{-0.6}\kern\dimen@\macc@nucleus\rel@kern{0.4}\kern\dimen@}%
      \advance\dimen@0.4\dimexpr\macc@kerna
      \let\final@kern#2%
      \ifdim\dimen@<\z@ \let\final@kern1\fi
      \if\final@kern1 \kern-\dimen@\fi
    \else
      \overline{\rel@kern{-0.6}\kern\dimen@#1}%
    \fi
  }%
  \macc@depth\@ne
  \let\math@bgroup\@empty \let\math@egroup\macc@set@skewchar
  \mathsurround\z@ \frozen@everymath{\mathgroup\macc@group\relax}%
  \macc@set@skewchar\relax
  \let\mathaccentV\macc@nested@a
  \if#31
    \macc@nested@a\relax111{#1}%
  \else
    \def\gobble@till@marker##1\endmarker{}%
    \futurelet\first@char\gobble@till@marker#1\endmarker
    \ifcat\noexpand\first@char A\else
      \def\first@char{}%
    \fi
    \macc@nested@a\relax111{\first@char}%
  \fi
  \endgroup
}
\makeatother

\newcommand{\expl}[2]{#1^{#2}}

\newcommand{\fin}{\mathbf{c}_f}

\newcommand*{\rom}[1]{\expandafter\@slowromancap\romannumeral #1@}
\newcommand{\lo}{\mathrm{LO}}
\newcommand{\tlo}{\mathrm{TLO}}
\newcommand{\rtlo}{\mathrm{CTLO}}
\newcommand{\explo}{\mathrm{EXPLO}}
\newcommand{\dlo}{\mathrm{dLO}}

\newcommand{\unbdd}[1]{#1^{\mathrm{u}}}

\newcommand{\bm}[1]{\mathbf{#1}}

\begin{document}
\begin{abstract}
It is well-known that every transitive linear order is exponentiable. However, is the converse true? This question was posed in Chapter 8 of the textbook titled ``Linear Orderings'' by Rosenstein. We define the class CTLO of cyclically transitive linear orders that properly contains the class of transitive linear orders, and show that all discrete unbounded orders in CTLO are exponentiable, thereby providing a negative answer to the question. The class CTLO is closely related to the class of transitive cyclic orders introduced by Droste, Giraudet and Macpherson. We also discuss the closure of subclasses of CTLO under products and iterated Hausdorff condensations. 
\end{abstract}
\maketitle

\section{Introduction}
Given ordinals $\alpha$ and $\beta$, the underlying set of the ordinal exponential $\alpha^\beta$ consists of all finite support functions from $\beta$ to $\alpha$, i.e., $\beta$-indexed sequences of elements of $\alpha$ that differ from the constant sequence with value $\bm 0$ in only finitely many indices. While generalizing the base and the exponent in this definition from ordinals to arbitrary linear orders, one needs to choose an element of the base order to play the role of $\bm 0$, thereby defining $(L,a)^{L'}$ for linear orders $L,L'$ and $a\in L$. The exponential defined this way enjoys all the nice arithmetic properties of the exponential of real numbers. However, in general, the order type(=order-isomorphism class) of the exponential depends on the choice of a base point. A linear order $L$ is said to be \emph{exponentiable} if the order types of its ordinal-indexed powers are independent of the choice of a base point. Some examples of exponentiable linear orders include the integers and the rationals with usual orders.

A linear order is said to be \emph{homogeneous} or \emph{transitive} if its automorphism group acts transitively on it. A transitive linear order is clearly exponentiable \cite[Proposition~8.17]{rosenstein1982linear}. Owing to the lack of examples of non-transitive yet exponentiable linear orders in the literature, Rosenstein declares the following question to be open.
\begin{ques}\cite[converse of Proposition~8.27, also p.162]{rosenstein1982linear}\label{ques 4.12}
Is every exponentiable linear order transitive?
\end{ques}

The main goal of this paper is to provide a negative answer to this question. We achieve this by defining the class $\rtlo$ of \emph{cyclically transitive} linear orders, that properly contains the class of transitive linear orders. Theorem \ref{cyctrans} shows that a linear order $L$ is cyclically transitive if and only if the associated cyclic order $\check L$ (in the sense of Droste et al. \cite{droste1995periodic}) is transitive. Thus, as a consequence of the classification of countable transitive cyclic orders by Campero-Arena and Truss \cite{campero20091}, we get a classification of all countable orders in $\rtlo$ (Corollary \ref{CTcor}).

We also investigate the closure of the class of cyclically transitive linear orders under various arithmetic operations. The class $\rtlo$ is closed under order reversal (Corollary \ref{ltloreversal}), cyclic equivalence (Corollary \ref{prop-ltlogen}) as well as multiplication on the left by transitive linear orders (Theorem \ref{prop5.5}). The class of discrete unbounded cyclically transitive linear orders is closed under multiplication (Theorem \ref{prop5.6}). Finally we show that $\rtlo$ is closed under iterated Hausdorff derivatives that are not bounded (Theorem \ref{ltlo-condens}). The main result of this paper (Theorem \ref{main-thm}) states that discrete unbounded cyclically transitive linear orders are exponentiable.

The rest of the paper is organised as follows. We recall in \S\ref{sect2} the basic definitions and notations associated with linear orders, including their arithmetic and iterated Hausdorff condensations. In \S\ref{sect4}, we define exponentiation of linear orders, and exponentiable as well as transitive linear orders. We also discuss their basic properties including a characterization of transitive linear orders by Morel (Theorem \ref{morel}). In \S\ref{sect5}, we introduce and discuss the properties of cyclically transitive linear orders, and subsequently prove the main theorem of the paper in \S\ref{section-last}. At the end of the paper, we pose some questions regarding cyclically transitive and exponentiable linear orders.

\section{Preliminaries}\label{sect2}
In this section we recall set up notations and recall some standard facts about linear orders. This material is taken from the standard, and perhaps only text \cite{rosenstein1982linear} on this topic--any notation or definition not explained in the paper can also be found in this text.

A linear order is a pair $(L,\leq)$, where $L$ is a possibly empty set and $\leq$ is a reflexive, anti-symmetric and transitive binary relation on $L$ where any two elements of $L$ are comparable under $\leq$. For brevity, we denote $(L,\leq)$ by $L$. For a linear order $L$, $L^*$ denotes the same underlying set with reverse order. Let $\lo$ denote the class of linear orders. An \emph{isomorphism} between linear orders is a bijective monotone map between them. The isomorphism class of a linear order is called its \emph{order type}.

A linear order $L$ is said to be \emph{unbounded} if for each $a\in L$ there are $b,c\in L$ such that $b<a<c$. For any class $\mathcal C$ of linear orders, the notation $\unbdd{\mathcal C}$ denotes its subclass consisting of unbounded linear orders. On the other hand, a linear order is \emph{bounded above} (resp. \emph{below}) if it has a greatest (resp. least) element. Say that a linear order is \emph{bounded} if it is bounded below as well as above. 

Let $\mathbb N$ denote the set of natural numbers, i.e., the set of non-negative integers. For $n\in\mathbb N$, let $\mathbf{n}$ denote the finite ordinal with $n$ elements. Let $\omega, \zeta$ and $\eta$ denote the order types of standard orderings on the sets $\mathbb N$, $\mathbb Z$ of integers and $\mathbb Q$ of rationals respectively.

A subset of a linear order is an \emph{interval} if it is a convex subset. Given elements $a,b$ in a linear order $L$, set $[a,b]:=\{c\in L\mid a\leq c\leq b\}\cup\{c\in L\mid b\leq c\leq a\}$.

There are two natural associative and non-commutative binary operations on linear orders, namely sum and product. Suppose $L_1,L_2$ are linear orders. Their \emph{sum} $L_1+L_2$ is defined as the disjoint union of the orders $L_1$ and $L_2$ where each element of $L_1$ is strictly below each element of $L_2$, whereas their \emph{product} $L_1 \cdot L_2$ consist of pairs of the form $(x,y)\in L_1\times L_2$ arranged in reverse dictionary/colex order, i.e., the second coordinate gets preference over the first one.

\begin{rem}
    For any linear orders $L_1, L_2 \text{ and } L_3$, we have $$L_1\cdot(L_2+L_3) \cong (L_1 \cdot L_2) + (L_1 \cdot L_3).$$ In other words, multiplication distributes over addition on the left. However, the right distributivity fails, as can be seen from $\omega \cong (\bm1 + \bm 1)\cdot \omega\not\cong \bm 1\cdot\omega+\bm1\cdot \omega\cong\omega\cdot\bm2$.
\end{rem}

We work with two important classes of linear order, namely discrete and dense.

A linear order $L$ is said to be \emph{discrete} if for every non-greatest element $a\in L$, $a$ has an immediate successor, denoted $a^+$, and for every non-smallest element $b \in L$, $b$ has an immediate predecessor, denoted $b^-$. For any class $\mathcal C$ of linear orders, the notation $\mathrm{d}\mathcal C$ denotes its subclass consisting of discrete linear orders. The order types $\bm n,\omega,\omega^*$ and $\zeta$ are discrete while $\eta$ is not.

A linear order $L$ is said to be \emph{dense} if for every $a<b$ in $L$, there exists $c \in L$ such that $a<c<b$. For any class $\mathcal C$ of linear orders, the notation $\mathrm{D}\mathcal C$ denotes its subclass consisting of discrete linear orders. A classic theorem of Cantor, whose proof technique is known as the back-and-forth method, states that the order type of a countable dense unbounded linear orders is $\eta$.

Given a linear order $L$, we define using transfinite recursion an equivalence relation $\sim_\gamma$ for each ordinal $\gamma$ and $a,b\in L$ as follows:
\begin{itemize}
    \item say $a\sim_0b$ if $a=b$;
    \item if $\gamma=\beta+1$ and $\sim_\beta$ is already recursively defined then say $a\sim_{\beta+1}b$ if the interval $[a,b]$ intersects only finitely many $\sim_\beta$-equivalence classes;
    \item if $\gamma$ is a limit ordinal and $\sim_\beta$ is recursively defined for each $\beta<\gamma$ then say $a\sim_\gamma b$ if $a\sim_\beta b$ for some $\beta<\gamma$. 
\end{itemize}
The equivalence classes for each $\sim_\gamma$ are intervals, so that $\fin^\gamma(L):=L/\sim_\gamma$ is a linear order. The order $\fin^\gamma(L)$ is called the $\gamma^{th}$ \emph{Hausdorff condensation} of $L$.

\begin{rem}\cite[Proposition~4.5]{AGRAWAL2023113639}
If $L\in\unbdd{\dlo}$ then $L\cong\zeta\cdot\fin(L)$.
\end{rem}

In a later section we need the following observation regarding sums of iterated Hausdorff condensations.
\begin{rem}\label{sumderivative}
Given an ordinal $\beta$ and a partition $L=L_1+L_2$, if either $\fin^\beta(L_1)$ is unbounded above or $\fin^\beta(L_2)$ is unbounded below then $\fin^\beta(L)=\fin^\beta(L_1)+\fin^\beta(L_2)$.
\end{rem}

We will also need a couple of facts about ordinals.
\begin{rem}\label{rem7}
    Any ordinal $\alpha$ can be uniquely written as $\alpha = \beta + \bm n$, where $\beta$ is either $\bm{0}$ or a limit ordinal, and $n$ is finite.
\end{rem}
\begin{rem}\label{cor6.2}
    For ordinals $\beta\leq\gamma$, there is a unique ordinal $\delta$ such that $\gamma=\beta+\delta$. We will denote this $\delta$ by $\gamma-\beta$. In fact, for $\beta\leq\alpha\leq\gamma$ we have $(\gamma-\beta) = (\alpha-\beta) + (\gamma-\alpha)$.
\end{rem}

\section{Exponentiable linear orders}\label{sect4}
In this section, we define exponentiation of linear orders and, consequently, exponentiable linear orders, and discuss their basic properties. We also define transitive linear orders and recall Morel's result on the characterization of transitive linear orders.

Recall that a pointed linear order $(L,a)$ is a non-empty linear order $L$ with an element $a\in L$.
\begin{defn}\label{expo}
    Given a pointed linear order $(L_1,a)$ and a linear order $L_2$, we define the \emph{exponential} of $(L_1,a)$ by $L_2$, denoted $(L_1,a)^{L_2}$, to be the set of all functions from $L_2$ to $L_1$ of finite support, i.e., taking value $a$ at all but finitely many elements in $L_2$, equipped with the following order: given distinct $f, g \in (L_1,a)^{L_2}$, say that $f < g$ if and only if, for the greatest $b \in L_2$ with $f(b) \neq g(b)$, we have $f(b) < g(b)$.
\end{defn}

It is possible to write a nice expression for ordinal exponents.
\begin{rem}\label{rem-rep}
Given a pointed linear order $(L,a)$, let the notations $L_{<a}, L_{>a}$ denote the sets $\{ b \in L \mid b < a \}$ and $\{ b \in L \mid a < b \}$ respectively. These sets can be used to find the following expression for $(L,a)^\alpha$ for an ordinal $\alpha$:
$$(L,a)^\alpha\cong \left(\operatornamewithlimits{\sum}\limits_{\beta \in \bm{\alpha}^{*}} (L,a)^{\beta} \cdot L_{<a} \right) + \bm 1 + \left(\operatornamewithlimits{\sum}\limits_{\beta \in \alpha} (L,a)^{\beta} \cdot L_{>a} \right).$$
\end{rem}

Exponentiation of linear orders satisfies the usual properties of exponentiation of real numbers.
\begin{prop}\label{prop 4.3}
    Given a pointed linear order $(L,a)$ and linear orders $L_1$,$L_2$, we have
    \begin{enumerate}
        \item $\expl{(L,a)}{L_1+L_2} \cong \expl{(L,a)}{L_1} \cdot \expl{(L,a)}{L_2}$,
        \item $\expl{(L,a)}{L_1 \cdot L_2} \cong \left(\expl{(L,a)}{L_1}\right)^{L_2}$.
    \end{enumerate}
\end{prop}
\begin{rem}\label{expminmax}
If $L$ has a least element, say $m$, then for any $L'$, the exponential $(L,m)^{L'}$ has a least element, namely the constant function with value $m$. Dually, if $L$ has a largest element, say $M$, then for any $L'$, $(L,M)^{L'}$ also has a largest element.
\end{rem}

\begin{prop}\label{unbddexp}
Given a pointed linear order $(L_1,a)$ and a linear order $L_2$ without a least element, we have $(L_1,a)^{L_2} \in \mathrm{D}\lo$.
\end{prop}
\begin{proof}
    If $L_1=\{a\}$ then the $(L_1,a)^{L_2}$ is a singleton, and hence trivially dense. Thus we assume that $L_1\neq\{a\}$. Let $f_1<f_2$ in $(L_1,a)^{L_2}$. Since $f_1,f_2$ are of finite support, we have smallest $b \in L_2$ with $f_1(b) \neq f_2(b)$. Since $L_2$ does not have a least element, we have some $b'<b$ in $L_2$ satisfying $f_1(b') = f_2(b') = a$.
    
    If $a$ is not the greatest element in $L_1$ then there exist some $a'>a$ in $L_1$. Let $f_3:L_2 \to L_1$ be the function that differs from $f_1$ only at the index $b'$, where it takes value $a'$. Then $f_3 \in (L_1,a)^{L_2}$ and $f_1<f_3<f_2$. 

    On the other hand, when $a$ is the greatest element in $L_1$, then we can choose some $a''<a$ in $L_1$. Let $f_4:L_2\to L_1$ be the function that differs from $f_2$ only at the index $b'$, where it takes value $a''$. Again, $f_4 \in (L_1,a)^{L_2}$ and $f_1<f_4<f_2$, thereby completing the proof that $(L_1,a)^{L_2} \in \mathrm{D}\lo$.
\end{proof}
\begin{cor}\label{cor 4.5}
Given a pointed linear order $(L_1,a)$ and a linear order $L_2$, there exist an ordinal $\alpha$ and a dense linear order $D$ such that $$(L_1,a)^{L_2} \cong (L_1,a)^\alpha \cdot D.$$
\end{cor}
\begin{proof}
    Every linear order $L_2$ can be uniquely expressed as $L' + L'',$ where $L'$ is a well-order and $L''$ has no least element. Since every well-order is isomorphic to a unique ordinal, we obtain an ordinal $\alpha$ such that $\alpha \cong L'$. Then Proposition \ref{prop 4.3}(1) yields $(L_1,a)^{L_2} \cong (L_1,a)^\alpha \cdot (L_1,a)^{L''}$. Finally, using Proposition \ref{unbddexp} we conclude $(L_1,a)^{L''} \cong D$ for some dense linear order $D$.
\end{proof}

In general, the order type of the exponential depends on the choice of a base point.
\begin{exmp} 
It follows from Remark \ref{rem-rep} that $(\omega,0)^\omega = \omega^\omega$ while $(\omega,1)^\omega = \operatornamewithlimits{\sum}\limits_{i \in \omega^*} \omega^i + \omega^\omega$. While the former is bounded below, the latter is unbounded, and hence they are not isomorphic.
\end{exmp}
In view of the above example, our focus is on the study of those linear orders in the base whose exponentiation is independent of the choice of base point. In view of Corollary \ref{cor 4.5}, we only ask for independence in the case of ordinal exponents.
\begin{defn}
    A linear order $L$ is said to be \emph{exponentiable} if for every ordinal $\alpha$ and every $a,b \in L$, we have $(L,a)^{\alpha} \cong (L,b)^{\alpha}$. We denote the class of exponentiable linear orders by $\explo$.
\end{defn}
If $L \in \explo$, then the notation $L^{L'}$ denotes the exponentiation of $L$ to the power $L'$ with respect to any choice of a base point in $L$.

\begin{rem}\label{rem8}
    For any linear order $L$ and given $a<b$ in $L$, the identity map yields an isomorphism $(L,a)^{\bm n} \cong (L,b)^{\bm n}$ for every finite $n$.
\end{rem}

Almost all exponentiable linear orders are unbounded.

\begin{prop}\label{expunbdd}
If $L\in\explo$ and $L$ has at least two elements then $L$ is unbounded.
\end{prop}
\begin{proof}
Suppose $L \in \explo$ and assume that $L$ has at least two elements. If $L$ has a least element, say $a$, then Remark \ref{expminmax} yields that $(L,a)^\alpha$ has a least element for any ordinal $\alpha$. If $b>a$ in $L$ and $\beta$ is an infinite ordinal then $L_{<b} \not\cong \bm 0$. Then thanks to Remark \ref{rem-rep}, we can see that $(L,b)^\beta$ does not have a least element. Hence, we see that $(L,a)^\beta \not\cong (L,b)^\beta$, which gives $L \not\in \explo$. A similar argument works if $L$ has a greatest element. This proves the result. 
\end{proof}

So far the only known exponentiable linear orders are those whose automorphism group acts transitively on them.
\begin{defn}
A linear order $L$ is said to be \emph{transitive} if for any $a,b\in L$, there exists an automorphism $\phi:L\to L$ satisfying $\phi(a) = b$. We denote the class of transitive linear orders by $\tlo$.
\end{defn}

Here is the promised result.
\begin{prop}\cite[Proposition 8.27]{rosenstein1982linear}\label{prop3.5}
$\tlo \subseteq \explo$.
\end{prop}
Question \ref{ques 4.12} asks if the converse of this statement is true.

\begin{rem}\label{transitiveunbdd}
A transitive linear order with at least two elements is unbounded.
\end{rem}

Transitivity of linear orders is preserved under exponentiation.
\begin{rem}
    Given linear orders $L \in \tlo$ and $L' \in \lo$, $L^{L'} \in \tlo$.
\end{rem}
This fact was used by Morel in her following characterization of transitive linear orders.
\begin{thm}\cite{zbMATH03212027}\label{morel}
Given a linear order $L$, we have $L\in\tlo$ if and only if $L\cong\mathbb{Z}^\gamma \cdot D$ for some ordinal $\gamma$ and dense $D\in\tlo$.
\end{thm}
\begin{cor}
    A countable linear order is transitive if and only if its order type is either $\zeta^\gamma$ or $\zeta^\gamma \cdot \eta$ for some countable ordinal $\gamma$.
\end{cor}

Recall that an endomorphism $\phi:L \to L$ of a linear order $L$ is said to be \emph{inflationary} if $x \leq \phi(x)$ for each $x \in L$. If $L\in\tlo$ and $a\leq b$ in $L$, then we can choose an inflationary automorphism that takes $a$ to $b$.
\begin{prop}\label{prop-weakinf}
Suppose $L \in \tlo$ and $a\leq b$ in $L$. Then there is an inflationary automorphism $\phi:L \to L$ such that $\phi(a) = b$.
\end{prop}
\begin{proof}
The conclusion clearly holds when $L=\bm1$, and hence thanks to Remark \ref{transitiveunbdd}, we may assume that $L \in \unbdd{\tlo}$. Suppose $a\leq b$ in $L$. If $a=b$ then the identity map is inflationary and takes $a$ to $b$. So we may assume that $a<b$. By transitivity there exist an automorphism $\phi$ of $L$ such that $\phi(a)=b$. Define $I_{a,b} := \{ c \mid \exists n \in \mathbb N \text{ with } \phi^{-n}(a) < c < \phi^{n}(a) \}$. It is readily seen that $I_{a,b}$ is an interval of $L$ satisfying $\phi(I_{a,b}) = I_{a,b}$. Then there is a partition $L = I_{-} + I_{a,b} + I_{+}$. The modified automorphism $\phi'$ of $L$ that takes $\phi$ on $I_{a,b}$ and identity elsewhere is inflationary and maps $a$ to $b$.
\end{proof}



\section{Cyclically transitive linear orders}\label{sect5}
In this section, we define an extension of the class $\tlo$, which we call the class of \emph{cyclically transitive linear orders}. We show that this class is closely related to the class of transitive cyclic orders (Theorem \ref{cyctrans}) that was introduced by Droste et al. \cite{droste1995periodic}. Thanks to Campero-Arena and Truss' classification of countable transitive cyclic orders (Theorem \ref{CT}), we give a classification of all countable cyclically transitive linear orders (Corollary \ref{CTcor}). We prove some closure properties for subclasses of the class $\rtlo$ under products (Theorems \ref{prop5.5},\ref{prop5.6}) and iterated Hausdorff condensations (Theorem \ref{ltlo-condens}).
\begin{defn}
    A linear order $L$ is said to be \emph{cyclically transitive} if for any $ a\leq b \text{ in } L $ there exist partitions $L=L_1+L_2= L_2'+L_1'$ with $a \in L_1, b \in L_1'$, and isomorphisms $F_1:L_1 \to L_1',F_2:L_2 \to L_2'$ satisfying $F_1(a) = b$. We denote the class of cyclically transitive linear orders by $\rtlo$.
\end{defn}

\begin{rem}\label{remtrivial}
    In the definition of cyclic transitivity, if $a=b$, then we can choose $L_1=L'_1:=L$, $L_2=L'_2:=\bm0$, $F_1$ to be the identity on $L$, and $F_2$ to be the empty map.
\end{rem}

\begin{rem}\label{rem3}
    We have $\tlo\subseteq\rtlo$ since for $L\in\tlo$ we can choose $L_1=L'_1:=L$ and $L_2=L'_2:=\bm{0}$ in the above definition. However, the reverse inclusion does not hold for $\bm{n}\in\rtlo\setminus\tlo$ for any $n\in\mathbb{N}$.
\end{rem}
\begin{rem}\label{rem2}
    If $L \in \rtlo$ is unbounded above (resp. below) then for any $a\leq b$ in $L$, the linear orders $L_1,L_1', L_2, L_2'$ guaranteed by the definition of $\rtlo$ are unbounded above (resp. below).
\end{rem}

The word ``cyclic'' in the definition of cyclic transitivity refers to its close connection with the notion of cyclic orders introduced by Droste et al.
\begin{defn}\cite{droste1995periodic}
Given a set $X$ and a ternary relation $R$ on $X$, say that the pair $(X,R)$ is a \emph{cyclic ordering} if
\begin{itemize}
    \item for all $x,y,z\in X$, $R(x,y,z)\iff R(y,z,x)\iff R(z,x,y)$;
    \item for all $a\in X$, if we define a binary relation $<_a$ on $X\setminus\{a\}$ by the rule: $x<_a y\iff R(a,x,y)$, then $<_a$ is a strict linear order.
\end{itemize}

Say that a cyclic ordering $(X,R)$ is \emph{transitive} if for any $a,b\in X$ there is an automorphism $\phi$ of $(X,R)$ satisfying $\phi(a)=b$, where an automorphism of $(X,R)$ is a bijective function $\phi:X\to X$ satisfying $R(x,y,z)\iff R(\phi(x),\phi(y),\phi(z))$ for all $x,y,z\in X$.
\end{defn}

Gluing the ends of a linear order $L$, we obtain a cyclic order $\check L:=(L,R)$ as per the following rule: for $a,b,c\in L$, $R(a,b,c)$ if and only if one of $a<b<c$, $b<c<a$ or $c<a<b$ holds.

The next result justifies the choice of the phrase ``cyclically transitive".
\begin{thm}\label{cyctrans}
A linear order $L$ is cyclically transitive if and only if the associated cyclic order $\check L$ is transitive.
\end{thm}

\begin{proof}
Suppose $L\in\rtlo$ and $a,b$ are any two elements of $L$. We give the construction of an automorphism $\phi$ of $\check L$ mapping $a$ to $b$. Without loss of generality, we may assume that $a\leq b$ in $L$. Cyclic transitivity of $L$ then yields partitions $L=L_1+L_2=L'_2+L'_1$ with $a\in L_1,b\in L'_1$, and isomorphisms $F_1:L_1\to L'_1$, $F_2:L_2\to L'_2$ satisfying $F_1(a)=b$. Then it is straightforward to verify that the map $\phi:L\to L$ defined as $\phi(x):=\begin{cases}F_1(x)& x\in L_1,\\F_2(x)& x\in L_2,\end{cases}$ is a required automorphism of $\check L$. 

Conversely, suppose that $\check L=(L,R)$ is transitive and $a\leq b$ in $L$. Then there is some automorphism $\phi$ of $\check L$ such that $\phi(a)=b$. Let $L_1:=\{x\in L\mid a\leq x\iff\phi(a)\leq\phi(x)\}$. It is routine to verify that if $x\leq y$ and $y\in L_1$ then $x\in L_1$. Similarly, we see that if $\phi(x)\leq z$ for some $x,z\in L$ then $z=\phi(y)$ for some $x\leq y$. Thus there exist partitions $L=L_1+L_2=\phi(L_2)+\phi(L_1)$ such that $a\in L_1$ and $b\in\phi(L_1)$. Finally, the restrictions of $\phi$ to $L_1$ and $L_2$ are required isomorphisms such that the former maps $a$ to $b$.
\end{proof}

We note some immediate consequences of the above equivalence.
\begin{cor}\label{ltloreversal}
The class $\rtlo$ is closed under order reversals.
\end{cor}
\begin{proof}
Note that $(X,R)$ is a cyclic order if and only if $(X,R^*)$ is a cyclic order, where $R^*:=\{(z,y,x)\in X^3\mid(x,y,z)\in R\}$. Clearly, $(X,R)$ is transitive if and only if $(X,R^*)$ is transitive. Finally, note that for any linear order $L$, we have $(\check{L})^*=(L^*)\check{}$ to complete the proof.
\end{proof}

The next definition captures the idea of cutting the cyclic order $\check L$ associated to a linear order $L$ at an arbitrary location.
\begin{defn}
Say that linear orders $L,L'$ are \emph{cyclically equivalent}, written $L\sim_c L'$, if there are partitions $L=L_1+L_2, L'=L'_2+L'_1$ and order isomorphisms $F_1:L_1\to L'_1, F_2:L_2\to L'_2$.
\end{defn}
It is easy to verify that $\sim_c$ is indeed an equivalence relation on the class $\lo$ of linear orders.

\begin{cor}\label{prop-ltlogen}
The class $\rtlo$ is closed under cyclically equivalent linear orders.
\end{cor}
\begin{proof}
If $L\sim_c L'$ then $\check L\cong\check{L'}$, and hence Theorem \ref{cyctrans} yields the result.
\end{proof}

Campero-Arena and Truss \cite{campero20091} classified all countable transitive cyclic orders.
\begin{thm}\cite[Theorem~2.12(i)]{campero20091}\label{CT}
Suppose $(X,R)$ is a countable cyclic order. Then $(X,R)$ is transitive if and only if $(X,R)$ is isomorphic to $\check L$, where the order type of $L$ is either $\zeta^\alpha\cdot\bm n$ or $\zeta^\alpha\cdot\eta$ for some countable ordinal $\alpha$ and finite $n$.
\end{thm}

This theorem gives us a classification of all countable cyclically transitive linear orders.
\begin{cor}\label{CTcor}
Suppose $L\in\rtlo$ is countable. Then $L\sim_c L'$, where the order type of $L'$ is either $\zeta^\alpha\cdot\bm n$ or $\zeta^\alpha\cdot\eta$ for some countable ordinal $\alpha$ and finite $n$.
\end{cor}

In the rest of this section, we investigate some arithmetic closure properties of $\rtlo$, and thereby help in constructing some new cyclically transitive linear orders from the existing ones.
\begin{thm}\label{prop5.5}
    If $L \in\tlo$ and $\overline{L} \in \rtlo$ then $L \cdot \overline{L} \in \rtlo$. Moreover, if $L$ is unbounded or discrete then so is $L \cdot \overline{L}$.
\end{thm}
\begin{proof}
    The proof of the second statement is obvious, so we only need to prove the first statement.
    
    Since $L\in\tlo$, Remark \ref{transitiveunbdd} yields that either $L\cong\bm1$ or $L\in\unbdd{\tlo}$. The conclusion is clear in the former case.
    
    Suppose $(a,\overline{a})\leq(b,\overline{b})$ in $L \cdot \overline{L}$, where $a,b \in L$ and $\overline{a},\overline{b} \in \overline{L}$. Then $\overline{a}\leq\overline{b}$ in $\overline{L}$, and hence by cyclic transitivity of $\overline{L}$, there are partitions $\overline{L} = \overline{L}_1 + \overline{L}_2 = \overline{L}_2' + \overline{L}_1'$, with $\overline{a} \in \overline{L}_1, \overline{b} \in \overline{L}_1'$, and bijective monotone maps $F_1:\overline{L}_1 \to \overline{L}_1'$ and $F_2:\overline{L}_2 \to \overline{L}_2'$ satisfying $F_1(\overline{a})=\overline{b}$. Moreover, since $L$ is transitive, there exists a bijective monotone map $ \phi:L \to L$ with $\phi(a)=b$.

    Choose partitions $L \cdot \overline{L} = L \cdot \overline{L}_1 + L \cdot \overline{L}_2 = L \cdot \overline{L}_2' + L \cdot \overline{L}_1'$ so that $(a,\overline a)\in L \cdot \overline{L}_1$ and $(b,\overline b)\in L \cdot \overline{L}_2'$. For $j=1,2$, let $G_j:L \cdot \overline{L}_j \to L \cdot \overline{L}_j'$ be a map defined by $$G_j(x,y) :=(\phi\times F_j)(x,y)= (\phi(x),F_j(y))\text{ for }x\in L,y \in \overline{L}_j.$$ The maps $G_j$ are bijective and monotone, and satisfy $G_1(a,\overline{a}) = (b, \overline{b})$. Therefore, $L \cdot \overline{L} \in \unbdd{\rtlo}$.
\end{proof}
\begin{thm}\label{prop5.6}
    If $L \in \rtlo$ and $\overline{L} \in \mathrm{d}\unbdd{\rtlo}$ then $L \cdot \overline{L} \in \unbdd{\rtlo}$. Hence, the class $\mathrm{d}\unbdd{\rtlo}$ is closed under products.
\end{thm}
\begin{proof}
    The proof of the second statement is obvious. Moreover, since $\overline{L}\in\unbdd{\lo}$ we have $L \cdot \overline{L}\in\unbdd{\lo}$. Thus, we only need to prove that $L \cdot \overline{L} \in \rtlo$.
    
    For any $c\leq d$ in $L$, by the definition of cyclically transitive linear orders, there are partitions $L = L_1 + L_2 = L_2' + L_1'$ with $c \in L_1, d \in L_1'$, and bijective monotone maps $F_1:L_1 \to L_1'$ and $F_2:L_2 \to L_2'$ satisfying $F_1(c)=d$.

    Suppose $(a,\overline{a})\leq(b,\overline{b})$ in $L \cdot \overline{L}$. Then $\overline{a} \leq \overline{b}$. Since $\overline{L}\in\rtlo$, there are partitions $\overline{L} = \overline{L}_1 + \overline{L}_2 = \overline{L}_2' + \overline{L}_1'$ with $\overline{a} \in \overline{L}_1, \overline{b} \in \overline{L}_1'$, and bijective monotone maps $\overline{F}_1:\overline{L}_1 \to \overline{L}_1'$ and $\overline{F}_2:\overline{L}_2 \to \overline{L}_2'$ satisfying $\overline{F}_1(\overline{a})=\overline{b}$. Choose partitions $L\cdot\overline{L} = L\cdot\overline{L}_1 + L\cdot \overline{L}_2 = L \cdot \overline{L}_2' + L \cdot \overline{L}_1'$ so that $(a,\overline{a})\in L\cdot\overline{L}_1$ and $(b,\overline{b})\in L\cdot\overline{L}'_1$.
    
    There are two cases:
    \begin{description}
    \item[Case $a \leq b$] Choose $c=a,d=b$ in the second paragraph of the proof to obtain $L = L_1 + L_2 = L_2' + L_1'$ with $a \in L_1, b \in L_1'$.   
    For $j=1,2$, define $G_j:L \cdot \overline{L}_j \to L \cdot \overline{L}_j'$, for $x\in L, y \in\overline{L}_j$, as
        $$G_j(x,y):=
        \begin{cases}
            (F_1(x),\overline{F}_j(y)) &\text{ if } x \in L_1,\\
            (F_2(x),\overline{F}_j(y)^+) &\text{ if } x \in L_2.
        \end{cases}$$
        Since $\overline{L}_j$ is discrete and unbounded, in view of Remark \ref{rem2}, the maps $G_j$ are bijective and monotone. Moreover, $G_1(a,\overline{a})=(b,\overline{b})$.
    \item[Case $b<a$] Then $\overline{a}<\overline{b}$. Choose $c=b$ and $d=a$ in the second paragraph of the proof to obtain $L = L_1 + L_2 = L_2' + L_1'$ with $b \in L_1, a \in L_1'$.
    
    For $j=1,2$, define $G_j:L \cdot \overline{L}_j \to L \cdot \overline{L}_j'$, for $x\in L,y\in\overline L_j$, as
            $$G_j(x,y): =
        \begin{cases}
            (F^{-1}_1(x),\overline{F}_j(y)) &\text{ if } x \in L_1',\\
            (F^{-1}_2(x),\overline{F}_j(y)^-) &\text{ if } x \in L_2'.
        \end{cases}$$
Since $\overline{L}_j$ is discrete and unbounded, in view of Remark \ref{rem2}, the maps $G_j$ are bijective and monotone. Moreover, $G_1(a,\overline{a})=(b,\overline{b})$.       
\end{description}
Therefore, we have completed the proof that $L\cdot\overline{L} \in \rtlo$.
\end{proof}

We use the results in this section to give some examples of cyclically transitive linear orders.
\begin{exmp}
We have $\omega+\omega^*\sim_c\zeta,\eta\sim_c\eta+\bm1\sim_c\bm1+\eta$, and hence all of these linear orders are in $\rtlo$ thanks to Corollary \ref{CTcor}. We also have $\omega+\zeta\cdot\eta+\omega^*\in\rtlo$.

Since $\bm n\in\rtlo$ for each $n\in\mathbb N$, Theorems \ref{morel} and \ref{prop5.5} together yield that $\zeta^\alpha\cdot D\cdot\bm n\in\rtlo$ for each ordinal $\alpha$ and transitive dense linear order $D$. 
\end{exmp}








To end this section, we see the effect of iterated Hausdorff condensation on $\rtlo$. 

\begin{thm}\label{ltlo-condens}
Suppose $L\in\rtlo$ is not bounded. Given an ordinal $\gamma$, if the condensation $\fin^\beta(L)$ is not bounded for each $\beta<\gamma$, then $\fin^\gamma(L)\in\rtlo$.
\end{thm}

\begin{proof}
We will prove the result by transfinite induction on $\gamma$.

If $\gamma=0$ then the conclusion is obvious.

If $\gamma=\beta+1$ for some ordinal $\beta$, then by the induction hypothesis we have that $\fin^\beta(L)\in\rtlo$. Moreover, the hypothesis guarantees that $\fin^\beta(L)$ is not bounded.

Let $\fin^\gamma(a)\leq\fin^\gamma(b)$ in $\fin^\gamma(L)$ for some $a,b\in L$. In view of Remark \ref{remtrivial}, we may assume that the inequality is strict. Then $\fin^\beta(a)<\fin^\beta(b)$ in $\fin^\beta(L)$. Since $\fin^\beta(L)$ is cyclically transitive, there are partitions $\fin^\beta(L)=L_1+L_2=L'_2+L'_1$ with $\fin^\beta(a)\in L_1,\fin^\beta(b)\in L_2$, and isomorphisms $F_1:L_1\to L'_1, F_2:L_2\to L'_2$ satisfying $F_1(\fin^\beta(a))=\fin^\beta(b)$. Since $\fin^\beta(L)$ is not bounded, either both $L_1$ and $L'_2$ are unbounded below or both $L'_1$ and $L_2$ are unbounded above. Using the isomorphisms $F_1$ and $F_2$, we conclude that either $L_1$ is unbounded above or $L_2$ is unbounded below, and a similar statement for the partition $L=L'_2+L'_1$. Thus Remark \ref{sumderivative} ensures that $\fin^\gamma(L)=\fin(L_1)+\fin(L_2)=\fin(L'_2)+\fin(L'_1)$ with $\fin^\gamma(a)\in\fin(L_1)$ and $\fin^\gamma(b)\in\fin(L'_1)$. Moreover, the isomorphisms $F_1,F_2$ induce appropriate isomorphisms. Thus $\fin^\gamma(L)\in\rtlo$.

If $\gamma$ is a limit ordinal, then by the hypothesis as well as the induction hypothesis we know for each $\beta<\gamma$ that $\fin^\beta(L)$ is cyclically transitive but not bounded.

Let $\fin^\gamma(a)\leq\fin^\gamma(b)$ in $\fin^\gamma(L)$ for some $a,b\in L$. In view of Remark \ref{remtrivial} we may assume that the inequality is strict. Then $\fin^\beta(a)<\fin^\beta(b)$ in $\fin^\beta(L)$ for each $\beta<\gamma$. In particular, $a<b$ in $L$. Since $L$ is cyclically transitive, there are partitions $L=L_1+L_2=L'_2+L'_1$ with $a\in L_1,b\in L_2$, and isomorphisms $F_1:L_1\to L'_1, F_2:L_2\to L'_2$ satisfying $F_1(a)=b$. For each $\beta<\gamma$, an argument similar to the successor case above shows that there are partitions $\fin^\beta(L)=\fin^\beta(L_1)+\fin^\beta(L_2)=\fin^\beta(L'_2)+\fin^\beta(L'_1)$ with induced isomorphisms $\fin^\beta(L_1)\cong\fin^\beta(L'_1)$ and $\fin^\beta(L_2)\cong\fin^\beta(L'_2)$. Thus by the definition of $\fin^\gamma$, we infer that for any $x\in L_1$ and $y\in L_2$, we have $\fin^\gamma(x)<\fin^\gamma(y)$ in $\fin^\gamma(L)$. Thus, $\fin^\gamma(L)=\fin^\gamma(L_1)+\fin^\gamma(L_2)$. Similarly, we can show that $\fin^\gamma(L)=\fin^\gamma(L'_2)+\fin^\gamma(L'_1)$. Moreover, the isomorphisms $F_1,F_2$ induce isomorphisms $\fin^\gamma(L_1)\cong\fin^\gamma(L'_1)$ and $\fin^\gamma(L_2)\cong\fin^\gamma(L'_2)$, where the former maps $\fin^\gamma(a)$ to $\fin^\gamma(b)$. Thus $\fin^\gamma(L)\in\rtlo$ in this case too, thereby completing the proof of the theorem.
\end{proof}

\begin{rem}
By developing the theory of iterated Hausdorff condensations for cyclic orders, the technique used in the proof of the above result will show that the class of transitive cyclic orders is closed under all iterated Hausdorff condensations.
\end{rem}

\section{Exponentiable subclass of cyclically transitive linear orders}\label{section-last}
The main goal of this section is to prove the theorem below that provides a class of exponentiable linear orders which are not necessarily transitive. Then we end with several questions regarding relationship between various classes of linear orders.

\begin{thm}\label{main-thm}
$\mathrm{d}\unbdd{\rtlo}\subseteq\explo$.
\end{thm}
\begin{proof}
    Suppose $L \in \mathrm{d}\unbdd{\rtlo}$. Fix some $a\leq b$ in $L$. Then, in view of Remark \ref{rem2}, there exist $L_1,L_1',L_2,L_2' \in \unbdd{\lo}$ with $a \in L_1, b \in L_1'$, and bijective monotone maps $F_1,F_2$ such that $F_1:L_1\to L_1'$ with $F_1(a)=b$ and $F_2:L_2\to L_2'$. In this proof, for any $c \in L$ and any ordinal $\beta$, the notation $c_{\beta}$ denotes the constant sequence indexed by $\beta$ with value $c$.  
    
    For each ordinal $\alpha$, we construct the following data using transfinite recursion:
    \begin{enumerate}
        \item partitions $(L,a)^\alpha=L(\alpha)_1+L(\alpha)_2$ and $(L,b)^\alpha=L(\alpha)'_2+L(\alpha)'_1$ satisfying $a_\alpha\in L(\alpha)_1$ and $b_\alpha\in L(\alpha)'_1$;
        \item whenever $\alpha=\beta+\bm1$ then, for $j=1,2$, $L(\alpha)_j= (L,a)^\beta\cdot L_j$ and $L(\alpha)'_j= (L,b)^\beta\cdot L'_j$;
        \item $L(\alpha)_2=L(\alpha)'_2=\bm0$ whenever $\alpha$ is $\bm0$ or a limit ordinal;
        \item isomorphisms $F(\alpha)_1:L(\alpha)_1\to L(\alpha)'_1$ and $F(\alpha)_2:L(\alpha)_2\to L(\alpha)'_2$ satisfying $F(\alpha)_1(a_\alpha)=b_\alpha$;
        \item whenever $\delta\leq\alpha$, $x\in L(\delta)_1$ and $z\in L(\delta)'_1$ then $(x,a_{\alpha-\delta})\in L(\alpha)_1$, $(z,b_{\alpha-\delta})\in L(\alpha)'_1$ and $F(\alpha)_1(x,a_{\alpha-\delta})=(F(\delta)_1(x),b_{\alpha-\delta})$.
    \end{enumerate}
    For $\delta\leq\alpha$, thanks to condition $(5)$ above, there will be embeddings $I_\delta^\alpha:L(\delta)_1\to L(\alpha)_1,\,x\mapsto(x,a_{\alpha-\delta})$, and $J_\delta^\alpha:L(\delta)'_1\to L(\alpha)'_1,\,z\mapsto(z,b_{\alpha-\delta})$ satisfying $F(\alpha)_1\circ I_\delta^\alpha=J_\delta^\alpha\circ F(\delta)_1$.
    
    There are three inductive steps in the construction of the inclusion-preserving partial isomorphism $F(\alpha)_1$ and the isomorphism $G(\alpha)$.
    \begin{description}
        \item[Case $\alpha=\bm0$] We can choose $L(\bm0)_1:=(L,a)^{\bm0}\cong\bm 0$, $L(\bm0)'_1:=(L,b)^{\bm0}\cong\bm0$, $L(\bm0)_2=L(\bm0)'_2:=\bm0$, and obvious isomorphisms $F(\bm0)_1$ and $F(\bm0)_2$.
        
        \item[Case $\alpha=\bm1$] It can be easily verified that 
        $F(\bm1)_1 := F_1$ and $F(\bm1)_2:=F_2$ satisfy the required conditions. 
        
        \item[Case $\alpha=\beta+\bm1$ for some limit ordinal $\beta$]
        By the inductive hypothesis, we have $(L,a)^\beta=L(\beta)_1$ and $(L,b)^\beta=L(\beta)'_1$. For $j=1,2$, set $L(\alpha)_j:= L(\beta)_1 \cdot L_j, L(\alpha)'_j := L(\beta)'_1 \cdot L'_j$ and $$F(\alpha)_j:=F(\beta)_1\times F_j.$$

        Being a product of two isomorphisms, $F(\alpha)_j$ is an isomorphism. Moreover, since $a\in L_1,b\in L'_1$, and by the inductive hypothesis we have $a_\beta\in L(\beta)_1$ and $b_\beta\in L(\beta)'_1$, we get that $a_\alpha\in L(\alpha)_1$ and $b_\alpha\in L(\alpha)'_1$ as required. Finally, thanks to Remark \ref{cor6.2}, in order to verify condition $(5)$, it is sufficient to note that whenever $x\in L(\beta)_1$, we have $(x,a)\in L(\alpha)_1$ (along with a similar statement for $L(\alpha)'_1$).
        
        \item[Case $\alpha = \beta + \mathbf{2}$ for some ordinal $\beta$] By the inductive hypothesis, we have partitions $(L,a)^{\beta+\bm1}=L(\beta+\bm1)_1 + L(\beta+\bm1)_2$ and $(L,b)^{\beta+\bm1}=L(\beta+\bm1)'_2 + L(\beta+\bm1)'_1$. For $j=1,2$, set $L(\alpha)_j:= L(\beta+\bm1)_1 \cdot L_j, L(\alpha)'_j := L(\beta+\bm1)'_1 \cdot L'_j$ and $$F(\alpha)_j(x,y) :=
        \begin{cases}
            (F(\beta+\bm{1})_1(x),F_j(y)) &\text{ if } x \in L(\beta+\bm1)_1,\\
            (F(\beta+\bm{1})_2(x),F_j(y)^+) &\text{ if } x \in L(\beta+\bm 1)_2.
        \end{cases}$$
        Since $L_j$ is discrete and unbounded, it can be verified that $F(\alpha)_j$ is an isomorphism. Moreover, since $a\in L_1,b\in L'_1$, and by the inductive hypothesis we have $a_{\beta+\bm1}\in L(\beta+\bm1)_1$ and $b_{\beta+\bm1}\in L(\beta+\bm1)'_1$, we get that $a_\alpha\in L(\alpha)_1$ and $b_\alpha\in L(\alpha)'_1$ as required. Finally, thanks to Remark \ref{cor6.2}, in order to verify condition $(5)$, it is sufficient to note that whenever $x\in L(\beta+\bm1)_1$, we have $(x,a)\in L(\alpha)_1$ (along with a similar statement for $L(\alpha)'_1$).

        \item[Case $\alpha$ is a limit ordinal] By the inductive hypothesis, for each $\delta\leq\gamma<\alpha$, we have injective monotone maps $I_\delta^\gamma:  L(\delta)_1\to L(\gamma)_1$ defined by $I_\delta^\gamma(x):=(x,a_{\gamma-\delta})$ and $J_\delta^\gamma: L(\delta)'_1\to L(\gamma)'_1$ defined by $J_\delta^\gamma(z):=(z,b_{\gamma-\delta})$. In view of Remark \ref{cor6.2}, both $\langle(L(\beta)_1)_{\beta<\alpha},(I_\beta^\gamma)_{\beta\leq\gamma<\alpha}\rangle$ and $\langle(L(\beta)'_1)_{\beta<\alpha},(J_\beta^\gamma)_{\beta\leq\gamma<\alpha}\rangle$ are directed systems. It is straightforward to verify that the directed union $\operatornamewithlimits{\bigcup}\limits_{\beta<\alpha} L(\beta)_1$ of the former system is indeed a linear order.
        
        Now we prove that $\operatornamewithlimits{\bigcup}\limits_{\beta<\alpha} L(\beta)_1=(L,a)^\alpha$. Let $x\in L(\beta)_1$. Then $I_\beta^{\beta+\bm1}(x)=(x,a)\in L(\beta+\bm1)_1=(L,a)^\beta\cdot L_1$. Hence $x\in(L,a)^\beta$. Thus we can define, for $\beta<\alpha$, an embedding $I'_\beta:L(\beta)_1\to (L,a)^\alpha$ by $I'_\beta(x):=(x,a_{\alpha-\beta})$. It is easy to see that $(I'_\beta)_{\beta<\alpha}$ form a cone under the directed system. Therefore, $\operatornamewithlimits{\bigcup}\limits_{\beta<\alpha} L(\beta)_1 \subseteq (L,a)^\alpha$. To see the reverse inclusion, choose any $x\in(L,a)^\alpha$. Then $x$ can be thought of as a finite support function from $\alpha$ to $L$ that takes value $a$ at all but finitely many places. Hence there exists $\beta<\alpha$ and $x'\in (L,a)^\beta$ such that $x=(x',a_{\alpha-\beta})$. Since $L(\beta+\bm1)_1=(L,a)^\beta\cdot L_1$ and $\alpha$ is a limit ordinal, $x'':=(x',a)\in L(\beta+\bm1)_1$ and $I'_{\beta+\bm1}(x'')=x$. Therefore, $\operatornamewithlimits{\bigcup}\limits_{\beta<\alpha} L(\beta)_1 = (L,a)^\alpha$. Set $L(\alpha)_1:=(L,a)^\alpha$ and $L(\alpha)_2:=\bm0$. Moreover, we can also relabel $I'_\beta$ as $I_\beta^\alpha:L(\beta)_1\to L(\alpha)_1$. Also define $I_\alpha^\alpha$ to be the identity on $L(\alpha)_1$.

        Similarly, we can show that $\operatornamewithlimits{\bigcup}\limits_{\beta<\alpha} L(\beta)'_1$ exists and $(L,b)^\alpha = \operatornamewithlimits{\bigcup}\limits_{\beta<\alpha} L(\beta)'_1$. We set $L(\alpha)'_1:=(L,b)^\alpha$ and $L(\alpha)'_2:=\bm0$ and define embeddings $J_\beta^\alpha:L(\beta)'_1\to L(\alpha)'_1$. Also define $J_\alpha^\alpha$ to be the identity on $L(\alpha)'_1$.

        It is readily verified for $\beta\leq\gamma\leq\alpha$, we have
        \begin{equation}\label{comp}
            I_\beta^\alpha=I_\gamma^\alpha\circ I_\beta^\gamma,\ J_\beta^\alpha=J_\gamma^\alpha\circ J_\beta^\gamma.
        \end{equation}

        Define $F(\alpha)_1: L(\alpha)_1 \to L(\alpha)'_1$ as follows. Let $x\in L(\alpha)_1$. Then thanks to the directed union, there is $\beta<\alpha$ and $x'\in L(\beta)_1$ such that $I_\beta^\alpha(x')=x$. Set $$F(\alpha)_1(x):=J_\beta^\alpha(F(\beta)_1(x')).$$
        
        \textbf{$F(\alpha)_1$ is well-defined:} Let $x''\in L(\gamma)_1$ satisfy $I_\gamma^\alpha(x'')=x$ for some $\gamma<\alpha$. Without loss of generality, we may assume that $\beta\leq\gamma$. Then $I_\beta^\gamma(x')=x''$. Hence, $$J_\gamma^\alpha(F(\gamma)_1(x''))=J_\gamma^\alpha(F(\gamma)_1(I_\beta^\gamma(x')))=J_\gamma^\alpha(J_\beta^\gamma(F(\beta)_1(x')))=J_\beta^\alpha(F(\beta)_1(x')),$$ where the second equality follows from condition $(5)$ of the inductive hypothesis, and the third equality follows from Equation \eqref{comp}.

        \textbf{$F(\alpha)_1$ is injective:} Suppose $x_1,x_2\in L(\alpha)_1$ satisfy $F(\alpha)_1(x_1)=F(\alpha)_1(x_2)$. Then, since $L(\alpha)_1=\operatornamewithlimits{\bigcup}\limits_{\gamma<\alpha} L(\gamma)_1$, there is some $\gamma<\alpha$ and elements $x'_1,x'_2\in L(\gamma)_1$ such that $J_\gamma^\alpha(F(\gamma)_1(x'_1))=J_\gamma^\alpha(F(\gamma)_1(x'_2))$. Since both $J_\gamma^\alpha$ and $F(\gamma)_1$ are injective, we conclude that $x'_1=x'_2$.

        \textbf{$F(\alpha)_1$ is surjective:} Let $z\in L(\alpha)'_1$. Then, since $L(\alpha)'_1=\operatornamewithlimits{\bigcup}\limits_{\gamma<\alpha} L(\gamma)'_1$, there is some $\gamma<\alpha$ and $y'\in L(\gamma)'_1$ such that $J_\gamma^\alpha(y')=y$. Since $F(\gamma)_1:L(\gamma)_1\to L(\gamma)'_1$ is an isomorphism, $F(\alpha)_1(I_\gamma^\alpha(F(\gamma)_1^{-1}(y')))=J_\gamma^\alpha(F(\gamma)_1(F(\gamma)_1^{-1}(y')))=J_\gamma^\alpha(y')=y$, as required.

        The bijective map $F(\alpha)_1:L(\alpha)_1\to L(\alpha)'_1$ is monotone since it is the directed union of the isomorphisms $(F(\gamma)_1)_{\gamma<\alpha}$. Our construction clearly shows that $F(\alpha)_1(a_\alpha)=b_\alpha$. Finally, note from the definition of the isomorphism $F(\alpha)_1$ that, for $\beta<\alpha$, we have $F(\alpha)_1\circ I_\beta^\alpha=J_\beta^\alpha\circ F(\beta)_1$.

    \end{description}

Now we are ready to describe the isomorphisms between $(L,a)^\alpha$ and $(L,b)^\alpha$ for all $\alpha$.

When $\alpha$ is a limit ordinal then $G(\alpha):=F(\alpha)_1:(L,a)^{\alpha} \to (L,b)^{\alpha}$ is an isomorphism.

On the other hand, if $\alpha$ is not a limit ordinal, then thanks to Remark \ref{rem7}, there exists a unique ordinal $\beta$ and a non-zero $n\in\mathbb N$ such that $\alpha=\beta+\bm n$. Then $(L,a)^\alpha=(L,a)^\beta\cdot(L,a)^{\bm n}$, and hence $G(\alpha)=G(\beta)\times G(\bm n)$ is the required isomorphism, where $G(\bm n)$ is the identity map as discussed in Remark \ref{rem8}. This completes the proof that $L\in\explo$.
\end{proof}

Since all infinite exponentiable linear orders are unbounded (Proposition \ref{expunbdd}), we can only hope to relax the discreteness hypothesis in the above theorem.
\begin{ques}
Is $\mathrm{D}\unbdd{\rtlo}\subseteq\explo$? 
\end{ques}

In fact, we do not even know the answer to the following question.
\begin{ques}
Is $\mathrm{D}\unbdd{\rtlo}\subseteq\tlo$?
\end{ques}

The most important question is the following extension of Question \ref{ques 4.12}. 
\begin{ques}
Is $\explo\subseteq\rtlo$?
\end{ques}




\section*{Declarations}
\subsection*{Ethical approval}
Not applicable

\subsection*{Declaration of interests}
None


\subsection*{Funding}
Not applicable

\subsection*{Availability of data and materials}
Not applicable

\printbibliography
\vspace{0.2in}
\noindent{}Mihir Mittal\\
Indian Institute of Technology Kanpur\\
Uttar Pradesh, India\\
Email 1: \texttt{mihirmittal24@gmail.com}\\
Email 2: \texttt{mihirm21@iitk.ac.in}

\vspace{0.2in}
\noindent{}Corresponding Author: Amit Kuber\\
Department of Mathematics and Statistics\\
Indian Institute of Technology Kanpur\\
Uttar Pradesh, India\\
Email 1: \texttt{askuber@iitk.ac.in}\\
Email 2: \texttt{expinfinity1@gmail.com}\\
Phone: (+91) 512 259 6721\\
Fax: (+91) 512 259 7500

\end{document}